\documentclass[reqno]{amsart}

\usepackage{enumerate}
\usepackage{amssymb}

\usepackage[colorlinks=true]{hyperref}
\hypersetup{urlcolor=blue, citecolor=cyan}
\hypersetup{citecolor=blue}

\usepackage[latin9]{inputenc}

\numberwithin{equation}{section}

\newtheorem{thm}{Theorem}[section]
\newtheorem{lem}[thm]{Lemma}

\newtheorem{cor}[thm]{Corollary}

\theoremstyle{definition}
\newtheorem{defn}[thm]{Definition}
\newtheorem{rem}[thm]{Remark}

\newcommand\R{{\mathbb R}}
\newcommand\C{{\mathbb C}}
\newcommand\N{{\mathbb N}}

\newcommand\Cz{{C_0(\R^N )}}
\newcommand\Czr{{C_0(\R )}}

\newcommand\Comp{{\mathrm{c}}}
\newcommand\Tma{T_{\mathrm{max}}}
\newcommand\Czu{{C_0(\R )}}
\newcommand\Ec[1]{e^{#1\Delta }}
\newcommand\Eqdef{\stackrel{\text{\rm\tiny def}}{=}}
\newcommand\NH{{\mathcal I}}
\newcommand\Fmu{{\mathcal F}^{-1}}

\newcommand\goto{\mathop{\longrightarrow}}

\newcommand\MScN[1]{\href{http://www.ams.org/mathscinet-getitem?mr=#1}{\nolinkurl{(#1)}}}
\newcommand\DOI[1]{\href{http://dx.doi.org/#1}{(doi: \nolinkurl{#1})}}
\newcommand\LINK[1]{\href{#1}{(link: \nolinkurl{#1})}}

\newcommand\DI{u_0 }
\newcommand\DIb{w_0 }

\begin{document}

\title{Non-regularity in H\"older and Sobolev spaces of solutions to the semilinear heat and Schr\"o\-din\-ger  equations}

\def\shorttitle{Non-regularity in H\"older and Sobolev spaces}

\author[T. Cazenave]{Thierry Cazenave$^1$}
\address{$^1$Universit\'e Pierre et Marie Curie \& CNRS, Laboratoire Jacques-Louis Lions,
B.C. 187, 4 place Jussieu, 75252 Paris Cedex 05, France}
\email{\href{mailto:thierry.cazenave@upmc.fr}{thierry.cazenave@upmc.fr}}

\author[F. Dickstein]{Fl\'avio Dickstein$^2$}
\address{$^2$Instituto de Matem\'atica, Universidade Federal do Rio de Janeiro, Caixa Postal 68530, 21944--970 Rio de Janeiro, R.J., Brazil}
\email{\href{mailto:flavio@labma.ufrj.br}{flavio@labma.ufrj.br}}

\author[F. B.~Weissler]{Fred B.~Weissler$^3$}
\address{$^3$Universit\'e Paris 13,  Sorbonne Paris  Cit\'e, CNRS UMR 7539 LAGA, 99 Avenue J.-B. Cl\'e\-ment, F-93430 Villetaneuse, France}
\email{\href{mailto:weissler@math.univ-paris13.fr}{weissler@math.univ-paris13.fr}}



\thanks{Research supported by the ``Brazilian-French Network in Mathematics"}
\thanks{Fl\'avio Dickstein  was partially supported by CNPq (Brasil), and by the Fondation Sciences Math\'ematiques de Paris.}

\begin{abstract}

In this paper we study the Cauchy problem for the semilinear heat and Schr\"o\-din\-ger equations, with the nonlinear term $ f ( u ) = \lambda   |u|^\alpha u$. 
We show that low regularity of $f$ (i.e., $\alpha >0$ but small) limits the regularity of any possible solution for a certain class of smooth initial data. 
We employ two different methods, which yield two different types of results. On the one hand, we consider the semilinear equation as a perturbation of the ODE $w_t= f(w)$. This yields in particular an optimal regularity result for the semilinear heat equation in H\"older spaces. In addition, this approach yields ill-posedness results for NLS in certain $H^s$ spaces, which depend on the smallness of $\alpha $ rather than the scaling properties of the equation. 
Our second method is to consider the semilinear equation as a perturbation of the linear equation via Duhamel's formula. This yields in particular that if $\alpha $ is sufficiently small and $N$ sufficiently large, then the nonlinear heat equation is ill-posed in $H^s (\R^N ) $ for all $s\ge 0$. 

\end{abstract}

\maketitle

%\begin{center} 
%\today
%\end{center} 

%\bigskip 
%\tableofcontents

\section{Introduction}

This paper is concerned with regularity of solutions of two well known and well studied semilinear evolution equations, the semilinear  heat equation
\begin{equation} \label{NLH}
\begin{cases} 
u_t= \Delta u+ \lambda  |u|^\alpha u \\
u(0, \cdot ) = \DI (\cdot ) 
\end{cases} 
\end{equation} 
and the semilinear Schr\"o\-din\-ger equation
\begin{equation} \label{NLSb}
\begin{cases} 
i u_t= \Delta u+ \lambda  |u|^\alpha u \\
u (0, \cdot ) = \DI (\cdot )
\end{cases} 
\end{equation} 
in $\R^N $, where $\alpha >0$ and $\lambda \in  \C$, $\lambda \not = 0$. 
More precisely, we allow the initial value $\DI $ to be infinitely smooth, and we study the loss of regularity due to the nonlinear term.
Therefore, we are particularly interested in small values of $\alpha >0$. Let $f(u)=  |u|^\alpha u$ with $0<\alpha <1$. As a point function $f$ is $C^1$ but not $C^2$. Formally, this might be considered an obstacle to the regularity of solutions of~\eqref{NLH} and~\eqref{NLSb}. 
Indeed, in order to prove the regularity of the solutions of~\eqref{NLH} and~\eqref{NLSb} (for instance by a fixed-point argument), one uses the regularity of the nonlinear term.
However, the relationship between the regularity of $f $ and the regularity of the solution is not a simple one.
Indeed, suppose $u $ is $C^2 (\R^N , \C )$,  $u(  x_0 )= 0$ and $\nabla u ( x_0)\not =  0$ for some $x_0\in \R^N $, then $ f(u) \not \in C^2 (\R^N , \C )$. 
On the other hand, for any reasonable initial value, for example in $\Cz$, the corresponding solution of~\eqref{NLH} will in fact be $C^2$ in space for $t>0$ by standard parabolic regularity. 
Thus the non-regularity of $f(u)$ does not immediately imply the non-regularity of $u$. 

The question of regularity is strongly related to the question of well-posedness. Recall that an evolution equation, such as~\eqref{NLH}  or~\eqref{NLSb}, is locally well-posed in a Banach space $X$ if for every $\DI \in X$ there exist $T>0$ and a solution $u\in C( [0,T], X)$ such that $u (0) =0$. In addition, the solution is required to be unique in some sense, not always in $C( [0,T], X)$, and is also required to depend continuously, again in some appropriate sense, on the initial value $\DI$. 
The key point for our purposes is that if $X$ is a positive order Sobolev space, whose elements have a certain degree of regularity, the resulting solution maintains this regularity. 

Specifically, if we wish to use a standard perturbation argument to prove that the Cauchy problem for either  equations~\eqref{NLH} or~\eqref{NLSb} is locally well-posed in $H^s (\R^N ) $ for some given $s>0$ we are confronted with two different requirements on $\alpha $. On the one hand, we need that the nonlinear term be controlled by the linear flow. This translates (formally) as the condition
\begin{equation} \label{fC1} 
\begin{cases} 
\displaystyle 0< \alpha \le  \frac {4} {N-2s}, &   s<  {N} / {2} \\
0<\alpha <\infty , &   s\ge   {N} / {2}.
\end{cases} 
\end{equation} 
On the other hand, in order to carry out the perturbation argument in $H^s  $, 
the nonlinear term must be sufficiently smooth. 
When $\alpha $ is not an even integer, then $\alpha $ must be  large enough so that $ f ( u ) =  |u|^\alpha u$ be sufficient regular.
In the case of the simplest perturbation argument requiring an estimate of $  |u|^\alpha u $ in $H^s (\R^N ) $, this leads to the condition
\begin{equation} \label{fC2} 
[s] <\alpha .
\end{equation} 
See e.g.~\cite{Ribaud, MolinetRY} for the heat equation, \cite{Kato2, Kato3, CWHs, Kato1, Pecher, FangH} for the Schr\"o\-din\-ger equation.

Since the first condition~\eqref{fC1} is related to scaling properties of the equation
(see Section~3.1 in~\cite{Tao}, and in particular the discussion p.~118), 
 it can be considered as natural. In fact, in some cases it is known that if this condition is not satisfied, then the problem is not well-posed in $H^s (\R^N ) $.  See for instance~\cite{MolinetRY, ChristCT, BurqGT, Carles, AlazardC}.
On the other hand, \eqref{fC2} might appear as a purely technical condition which one should be able to remove by a more appropriate argument. 
Indeed, one can sometimes improve condition~\eqref{fC2}  by using the fact that one time derivative is like two space derivatives, but we are still left with the condition 
\begin{equation} \label{fC3} 
[s] <2\alpha .
\end{equation} 
See~\cite{Kato3, Pecher, FangH}.

The purpose of this paper, as opposed to the above cited papers, is to show that in certain cases  ``technical" restrictions such as~\eqref{fC2} and~\eqref{fC3} are not purely technical, but impose genuine limitations on the regularity of the solution.
More precisely, we show that condition~\eqref{fC2} is not always sufficient to imply local well-posedness of~\eqref{NLH}  and~\eqref{NLSb} in $H^s$. (See Remarks~\ref{eRR1b} and~\ref{eRR1c}, and Theorem~\ref{eIP1}.) In fact, we prove under various circumstances that there exist initial values $\DI \in C^\infty _\Comp (\R^N )$ for which~\eqref{NLH}  or~\eqref{NLSb} cannot have a local solution with a certain degree of regularity. To our knowledge, there are no previous results of this type.

Our first result concerns the nonlinear heat equation~\eqref{NLH},  and is in fact optimal.
We recall that the Cauchy problem~\eqref{NLH} is locally well-posed in $\Czr$, i.e. for any $\DI \in \Czr$, there exist a maximal existence time $\Tma >0$ and a unique  solution  $u\in C([0, \Tma) , \Czr) $ of~\eqref{NLH}.
Let $0<\alpha <1$, let  $\DI$ be smooth, and let $u\in C([0, \Tma) , \Cz)$ be the resulting, maximal solution of~\eqref{NLH}. It is known that, given any $0< T< \Tma$, 
$\partial _t u$, $\nabla \partial _t u$, and all space-derivatives of $u$  of order up to  $ 3$ belong to $C([0,T] \times \R^N )$. Furthermore, the spatial derivatives of order $3$ are $\alpha $-H\"older continuous, i.e.
\begin{equation} \label{fP1} 
\sup 
 _{\substack{ 0\le t\le T\\    |m |=3}}
 |   \partial _x^m u  (t ) |_{ \alpha }   < \infty 
\end{equation} 
where $  | w |_\ell $ is defined by
\begin{equation} \label{fHO1} 
 | w |_\ell \Eqdef \sup  _{ x\not = y } \frac { | u (x) - u (y) |} { |x - y|^\ell }
\end{equation} 
for $\ell >0$ and $w\in C (\R^N )$. See Theorem~\ref{Holder} in the Appendix for a precise statement. The theorem below shows that the $\alpha $-H\"older continuity cannot in general be improved in the sense that one cannot replace $\alpha $ by $\beta >\alpha $.

\begin{thm} \label{eT3} 
Let $0<\alpha <1$ and $\lambda \in \C  \setminus \{0\}$. There exists an initial value $\DI \in C^\infty _\Comp (\R^N )$ 
such that the corresponding maximal solution $u\in C([0, \Tma) , \Cz)$ of~\eqref{NLH} is
three times continuously differentiable with respect to the space variable and
\begin{equation} \label{feT3} 
\int _s ^t  | \partial _y \Delta  u(\sigma )  |_\beta  =+\infty 
\end{equation} 
for all $\alpha <\beta \le 1$ and all $0\le s<t < \Tma$ with $t$ sufficiently small.
\end{thm} 

Theorem~\ref{eT3} has the immediate following corollary, by using Sobolev's embedding theorem (see~\eqref{fSob2}).

\begin{cor} \label{eT1} 
Let $0<\alpha <1$. There exists an initial value $\DI \in C^\infty _\Comp (\R^N )$ 
such that the corresponding maximal solution $u\in C([0, \Tma) , \Cz)$ of~\eqref{NLH}
does not belong to $L^1((0,T) , H^{s, p} (\R^N ) )$ if $s> 3+ \frac {N} {p} + \alpha $, $1<p<\infty $, and $0< T< \Tma$.
\end{cor} 

\begin{rem} \label{eRR1b} 
The initial value $\DI \in C^\infty _\Comp (\R^N )$  can be chosen arbitrarily small (in any space).
See Remark~\ref{eRemSD}. 
Corollary~\ref{eT1} therefore implies that $s\le  3+ \frac {N} {p} + \alpha $, $1<p<\infty $ is a necessary condition for~\eqref{NLH}  to be locally well posed in $H^{s, p} (\R^N ) )$, even in an arbitrarily small ball.
\end{rem} 

\begin{rem} \label{eRR1} 
We should observe that formula~\eqref{feT3} does not imply that $ | \partial _y \Delta  u(t )  |_\beta$ is infinite for any given value of $0< t <\Tma$. On the other hand, it is stronger than saying that $\sup  _{ 0\le t\le T }
 |   \partial _y \Delta   u  (t ) |_{ \beta  }   < \infty $. 
 Similarly, Corollary~\ref{eT1} does not guarantee that $ u(t) \not \in H^s (\R^N )  $ for any given value of $0< t <\Tma$.
\end{rem} 

Before stating our next result, we make some comments on the proof of Theorem~\ref{eT3}.
The key idea is to consider equation~\eqref{NLH} as a perturbation of the ordinary differential equation 
\begin{equation} \label{fode1} 
w_t = \lambda  |w|^\alpha w 
\end{equation} 
with the same initial condition $w (0, \cdot )= \DI (\cdot )$. 
As we shall see by a straightforward calculation (see Section~\ref{sODE}), equation~\eqref{fode1}  produces a loss of spatial regularity. For example, in dimension $N=1$, if $\DI (x) =x$ in a neighborhood of $0$, then the resulting solution $w (t, x)$ of~\eqref{fode1}  will not be twice differentiable at $x=0$ for  $t>0$.
Moreover, for the perturbed equation 
\begin{equation}  \label{fode2} 
w_t = \lambda  |w|^\alpha w +h 
\end{equation} 
where $h$ is sufficiently smooth, the same loss of regularity occurs. 
(See Theorem~\ref{eODE2}.)
Let now $u$ be a solution of the nonlinear heat equation~\eqref{NLH}, and set $h= \Delta u$. 
It follows that $u_t= \lambda  |u|^\alpha u +h$. 
Thus we see that  for appropriate $\DI$,  if $h$  is sufficiently smooth, then $u (t)$ is not $C^2$ in space for small $t>0$. Since we know $u(t)$ is $C^2$ for $t>0$, this implies that $h = \Delta u$ is not too regular. Applying the precise regularity statement of Theorem~\ref{eODE2} gives the conclusion of 
Theorem~\ref{eT3}.

It turns out that the same arguments can be used to prove ill-posedness for the nonlinear Schr\"o\-din\-ger equation~\eqref{NLSb}. 
This yields the following analogue of Corollary~\ref{eT1}. 

\begin{thm} \label{eNLS} 
Let $0<\alpha <1$, $\lambda \in \C \setminus \{ 0\} $ and suppose $s> 3+ \frac {N} {2} + \alpha $. There exists $\DI \in C^\infty _\Comp (\R^N )$ such that there is no $T>0$ for which there exists a solution $u\in C([0,T], H^s (\R^N ) )$ of~\eqref{NLSb}.
\end{thm} 

\begin{rem} \label{eRR1c} 
Theorem~\ref{eNLS} implies that $s\le  3+ \frac {N} {2} + \alpha $ is a necessary condition for~\eqref{NLSb}  to be locally well posed in $H^{s} (\R^N ) )$, 
 even in an arbitrarily small ball. See Remark~\ref{eRemSD}.
\end{rem} 

Theorem~\ref{eNLS} turns out to be a specific case of  an analogous result for the complex Ginzburg-Landau equation, see Theorem~\ref{eGL} below.

As pointed out in Remark~\ref{eRR1}, Theorem~\ref{eT3} and Corollary~\ref{eT1} do not guarantee the lack of spatial regularity of the solution $u$ of~\eqref{NLH}  at any fixed $t>0$. 
The following theorem gives an example of loss of spatial regularity for every $t>0$. 

\begin{thm} \label{eRN2} 
Let $0<\alpha <2$ and $\lambda \in \C \setminus \{0\}$.
There exists an initial value $\DI \in C^\infty _\Comp (\R^N )$ 
such that the corresponding maximal solution $u\in C([0, \Tma) , \Cz)$ of~\eqref{NLH} satisfies
$u(t) \not \in H^{s,p} (\R^N )$ for all sufficiently small $0 <t <\Tma$  if $1< p<\infty $ and $s > 5 + \frac {1} {p}$.
\end{thm} 

\begin{rem} \label{eReRN2} 
As observed for previous results, the initial value $\DI \in C^\infty _\Comp (\R^N )$ in Theorem~\ref{eRN2}  can be chosen arbitrarily small (in any space).  See Remark~\ref{eRemSDb}.
\end{rem} 

Unlike the proof of Theorem~\ref{eT1}, the proof of Theorem~\ref{eRN2} treats the nonlinear term $ |u|^\alpha u$ as a perturbation of the linear heat equation, via the standard Duhamel formula. 
More precisely, for appropriate initial values we show that the integral term 
\begin{equation} \label{fDU1} 
{\mathcal I} = \int _0^t e^{(t-s) \Delta }  |u(s)|^\alpha u(s) \, ds
\end{equation} 
can never be in $H^{s,p} (\R^N )$ if $s > 5 + \frac {1} {p}$.
One key idea in the proof is to express $ | u(t,x', y) |^\alpha u(t,x', y) = C(x') \gamma (t)  |y|^\alpha y +  \widetilde{w} (t,x', y)$ where $x'\in \R^{N-1}$ and $y\in \R$, with $ |  \widetilde{w} (t,x', y) |\le C  |y|^{\alpha +2}$. This decomposition enables us to explicitly compute $\partial ^5_y e^{\varepsilon \Delta } {\mathcal I} $ at $y=0$, which (if $\alpha <2$) goes to $\infty $ as $\varepsilon \downarrow 0$, uniformly for $x'$ in a neighborhood of $0$. This shows that $  {\mathcal I}  $ cannot be $C^5$ with respect to $y$, and the result then follows from the one-dimensional Sobolev embedding theorem.
We insist on this last point; since the proof is based on a one-dimensional argument, the condition on $s,p$ in the statement of Theorem~\ref{eT1} is independent of the space dimension $N$. 

On the other hand, in Corollary~\ref{eT1} the condition on $s,p$ does depend on the space dimension, since we deduce the result from Theorem~\ref{eT3} by the $N$-dimensional Sobolev embedding theorem.
This is perhaps only a technical problem. Indeed, the proof of Theorem~\ref{eT3} is also based on a one-dimensional argument. However, the structure of that proof, via an argument by contradiction, does not seem to allow the application of the one-dimensional Sobolev embedding theorem. 

For our last result, we introduce a very weak notion of local well-posedness for small data, which is weaker than the general notion described earlier in the introduction. 
Recall that~\eqref{NLH} is locally well-posed in $\Cz$, and  $\Tma (\DI ) $ is the maximal existence time of the solution corresponding to the initial value $\DI$.

\begin{defn} \label{eDhs} 
Let $s\ge 0$, $\alpha >0$ and $\lambda \in \C$. 
We say that~\eqref{NLH} is locally well posed for small data in $H^s (\R^N ) $ if there exist $\delta , T>0$ such that if $\DI \in C^\infty _\Comp (\R^N ) $ and $ \| \DI \| _{ H^s } \le \delta $, then the corresponding solution $u \in C([0, \Tma ), \Cz )$ of~\eqref{NLH} satisfies $\Tma (\DI) \ge T$ and $u(t) \in H^s (\R^N ) $ for all $0\le t\le T$. 
\end{defn} 

\begin{thm} \label{eIP1} 
Let $0< \alpha < 2$, and $\lambda \in\R$ with $\lambda >0$. If
\begin{equation} \label{fCOR1}  
N > 11 + \frac {4} {\alpha }.
\end{equation} 
then for every $s\ge 0$, the Cauchy problem~\eqref{NLH} is not locally well posed for small data in $H^s (\R^N ) $.
\end{thm} 

The rest of the paper is organized as follows. 
We recall below the definitions of the various function spaces we use, and certain of their properties. 
In Section~\ref{sODE} we study regularity of solutions to the ordinary differential equation~\eqref{fode1}, and to the perturbed equation~\eqref{fode2}.  In particular, we show (Theorem~\ref{eODE2}) that if $h$ is sufficiently smooth then~\eqref{fode1}  produces a singularity for a certain class of smooth data. In Section~\ref{sAppl} we apply this result to prove Theorems~\ref{eT3} and~\ref{eNLS}, as well as a similar result for a complex Ginzburg-Landau equation (Theorem~\ref{eGL}). 
In Section~\ref{sLast} we prove Theorems~\ref{eRN2} and~\ref{eIP1}.

One final remark about our results. Throughout this paper, we have considered small values of $\alpha $, either $0<\alpha <1$ or $0<\alpha <2$. 
It is likely that analogous results can be proved for larger values of $\alpha $.

\medskip 
\noindent {\bf Notation and and function spaces.} Throughout this paper, we consider function spaces of complex-valued functions.

$L^p (\R^N ) $, for $1\le p\le \infty $, is the usual Lebesgue space, with norm $  \| \cdot  \| _{ L^p } $. 
We denote by $\Cz$ the space of continuous functions on $\R^N $ that vanish at infinity, equipped with the sup norm. $H^{s, p} (\R^N ) $ and $H^s( \R^N ) = H^{s, 2} (\R^N ) $, for $s\ge 0$ and $1< p <\infty $ are the usual Sobolev spaces, and the corresponding norms are denoted by $  \| \cdot  \| _{ H^s } $ and $  \| \cdot  \| _{ H^{s, p} } $. 
In particular, 
\begin{equation} \label{fSob3} 
 \| u \| _{ H^{s, p} }  =  \| \Fmu [ (1+ |\xi |^2)^{\frac {s} {2}}  \widehat{u}  ] \|  _{ L^p }
\end{equation} 
and $ \| u \| _{ H^{s, p} }  \approx \sum_{  |\ell| \le s }  \| \partial ^\alpha u \| _{ L^p }$ if $s$ is an integer. (See e.g.~\cite[Theorem~6.2.3]{BerghL}.)

In the proof of Theorem~\ref{eRN2}, we use the property that if 
$u= u (x_1, x_2)$ with $x_1 \in \R^m$, $x_2 \in \R^n$, and if  $1< p<\infty $ and $s\ge 0$, then
\begin{equation} \label{fSob1} 
  \| u \| _{ L^p _{ x_1 } (\R^m, H^{s, p}( \R^n _{ x_2 }) ) } \le   C \| u \| _{ H^{s, p} (\R^{m+n} ) }  .
\end{equation}  
Inequality~\eqref{fSob1} with $C=1$ is immediate when $s$ is an integer. (The left-hand side has fewer terms than the right-hand side.)
The general case follows by complex interpolation. Indeed, suppose $s$ is not an integer,
 fix two integers $0\le s_0 < s < s_1$ and let $0< \theta <1$ be defined by $s= (1-\theta ) s_0 +\theta s_1$. It follows that 
 $H^{s,p} (\R^{m+n} ) = (H^{s_0, p} (\R^{m+n} ) , H^{s_1,p} (\R^{m+n} ) ) _{ [\theta ] }$
 and $H^{s,p} (\R^{n} ) = (H^{s_0, p} (\R^{n} ) , H^{s_1,p} (\R^{n} ) ) _{ [\theta ] }$. (See~\cite[Theorem~6.4.5]{BerghL}.)
 This last property implies that 
 \begin{equation*} 
 L^p ( \R^m,  H^{s,p} (\R^{n} )) = (L^p (\R^m, H^{s_0, p} (\R^{n} )) , L^p (\R^m, H^{s_1,p} (\R^{n} )) ) _{ [\theta ] }. 
 \end{equation*} 
 (See~\cite[Theorem~5.1.2]{BerghL}.)
Estimate~\eqref{fSob1} now follows by complex interpolation between the estimates for $s=s_0$ and $s=s_1$. 

We will use Sobolev's embedding into H\"older spaces. 
Recall definition~\eqref{fHO1}.
Given any $j\in \N$ and $0<\ell < 1$, the H\"older space $C^{j,\ell} (  \overline{\R^N }  )$ is the space of functions $u$ whose derivatives of order $\le j$ are all bounded and uniformly continuous, and such that $  | \partial ^\gamma  u |_\ell <\infty  $ for all multi-indices $\gamma $ with  $ |\gamma |=j$. 
$C^{j,\ell} (  \overline{\R^N }  )$ is a Banach space when equipped with the norm $  \|u \| _{ W^{j,\infty } } + \sum_{  |\gamma |=j }  |\partial ^\gamma  u |_\ell $. (See e.g.~\cite[Definition~1.29, p.~10]{AdamsF}.)
Given  $j\in \N$, $0< \ell <1$ and  $ s \in ( \ell+j,  \ell +j +N) $, 
It follows that 
\begin{equation} \label{fSob2} 
H^{ s , p(s) } (\R^N ) \hookrightarrow C^{j,\ell} (  \overline{\R^N }  )
\end{equation} 
where $p (s) = \frac {N} {s-j-\ell }\in (1,\infty )$.
Indeed, we may assume $j=0$, as the general case follows by iteration. 
Suppose first $s \ge 1$. 
Since $s  - \frac {N} { p(s)} =\ell  = 1  - \frac {N} {p(1 )}$, it follows from~\cite[Theorem~6.5.1, p.~153]{BerghL} that $H^{s,p (s) } (\R^N ) \hookrightarrow H^{1 ,p (1  ) } (\R^N ) $. The result now follows from the embedding $H^{ 1  , p(1   ) } (\R^N ) \hookrightarrow C^{0,\ell} (  \overline{\R^N }  )$. (See~\cite[Theorem~4.12 Part~II, p.~85]{AdamsF}.)
Let now $\ell    <s<1  $, and note that $H^{s,p (s) } (\R^N ) \hookrightarrow B^{s, p (s) }_\infty  (\R^N ) $, see~\cite[Theorem~6.2.4, p.~142]{BerghL}.
Setting $m=1$, we have $m-1 < \frac {N} {p(s)} < s < m$, and the result follows from 
the embedding $B^{s, p (s) }_\infty  (\R^N )\hookrightarrow C^{0,\ell} (  \overline{\R^N }  ) $.
(See~\cite[Theorem~7.37, p.~233]{AdamsF}.)

We denote by $(e^{t\Delta }) _{ t\ge 0 }$ the heat semigroup on $\R^N $, and we recall that $e^{t \Delta }$ is a contraction on $L^p (\R^N ) $ for all $1\le p\le \infty $. 
Using~\eqref{fSob3}, it follows immediately that
$(e^{t\Delta }) _{ t\ge 0 }$ is also a contraction semigroup on $H^{s, p}  (\R^N ) $ for all $s\ge 0$ and $1 < p<\infty $. 

\section{Spatial singularities and ODEs} \label{sODE} 

In this section we study how a certain class of ordinary differential equations lead to loss of regularity. More precisely, we consider equations~\eqref{fode1}  and~\eqref{fode2}, which are ODEs with respect to time, as acting on functions depending  on a space variable $  |x| \le 1$. In particular, the initial value $w (0, \cdot ) = \DIb (\cdot ) $ is a function $\DIb : [-1, 1] \to \C$. 
We wish to study the spatial regularity of $w (t, \cdot )$ as compared to the spatial regularity of $\DIb$. This is a different phenomenon than finite-time blowup. 
For example, consider the ODE-initial value problem
\begin{equation} \label{fO1} 
\begin{cases} 
w_t =\lambda   |w|^\alpha w \\ w( 0, x) = x,
\end{cases} 
\end{equation} 
with $ \lambda  \in \C$, $\lambda \not = 0$ and $  |x|\le 1$. 
If $\Re \lambda \not = 0$, then the solution of~\eqref{fO1} is given by
\begin{equation*} 
w(t,x) = \frac { x} {(1 - \alpha   t  | x |^\alpha \Re \lambda   )^{\frac {\lambda } {\alpha \Re \lambda }  }}.
\end{equation*} 
It follows that
\begin{gather*} 
w_x (t,x)=  (1 + i \alpha  t  | x |^\alpha \Im \lambda  )
 (1 - \alpha  t  | x |^\alpha \Re \lambda  )^{- \frac {\lambda +\alpha \Re \lambda } {\alpha \Re \lambda } } \\
w _{ xx } (t,x)  =  \alpha  t \frac { |x|^\alpha } {x}  (1 - \alpha  t  | x |^\alpha \Re \lambda  )^{- \frac {\lambda + 2\alpha \Re \lambda } {\alpha }  }   [  \lambda +\alpha \Re \lambda  + i\alpha  (\Re \lambda ) (1+ \lambda t |x|^\alpha ) ]
\end{gather*} 
for $x\not = 0$, as long as these formulas make sense. 
If $\Re \lambda  =0$, then the solution of~\eqref{fO1} is given by
\begin{equation*} 
w(t,x) = \exp ( i   t  | x |^\alpha \Im \lambda  ) \DIb (x) 
\end{equation*} 
and so,
\begin{gather*}
w_x (t,x)=   (1 + i \alpha  t  | x |^\alpha  \Im \lambda  ) e^{  i  t  | x |^\alpha  \Im \lambda  } \\
w _{ xx } (t,x)=  i \alpha   t  ( \Im \lambda  ) (1 + \alpha  +  i\alpha  t  | x |^\alpha  \Im \lambda  )  \frac { |x|^\alpha } {x}  e^{ i  t  | x |^\alpha  \Im \lambda  } 
\end{gather*} 
for $t\ge 0$ and $x\not = 0$. 
In both cases,  $w(t)$ is $C^1$ in $[-1,1]$ as long as it exists. However, if $\alpha <1$, we see that $w(t)$ fails to be twice differentiable at $x=0$, for  $t>0$. 

Somewhat surprisingly, it turns out that this loss of spatial regularity also occurs for regular perturbations of~\eqref{fO1}, as the following theorem shows.

\begin{thm} \label{eODE2} 
Let $0 <\alpha <1$, $\lambda \in \C \setminus \{ 0\}$,  $T>0$, $\DIb \in C^2 ([-1, 1], \C )$ and $h \in C ([0, T] \times [-1, 1] , \C )$ such that $\partial _y h \in C ([0, T] \times [-1, 1] , \C )$.
Suppose further that $\DIb (0)= 0$ and $h (t, 0 ) = 0$ for $0\le t\le T$.  By possibly assuming that $T>0$ is smaller, it follows that there exists a solution $ w \in C ^1 ([0, T] \times [-1, 1] , \C )$ of the equation
\begin{equation} \label{fODE4} 
\begin{cases} 
w _t =  \lambda   |w|^\alpha w + h  & 0\le t\le T, -1 \le y\le 1\\
w (0, y)= \DIb (y) & -1 \le y\le 1
\end{cases} 
\end{equation} 
If $\DIb ' (0) \not = 0$ and 
\begin{equation} \label{fODE5} 
\int _0^\tau  \sup  _{0 <  |y| \le 1}
\frac { | \partial _y h (t, y) - \partial _y h (t,0) |} { |y|^\beta }< \infty 
\end{equation} 
for some $0<\tau \le T$ and $\beta >\alpha $, then
\begin{equation}  \label{fODE6} 
\liminf _{   |y| \downarrow 0 } \frac { | \partial _y w (t, y) -\partial _y w (t,0) |} {  |y|^\alpha } >0
\end{equation} 
for all  sufficiently small $t>0$.
In particular, $w (t, \cdot )$ is not twice differentiable at $y=0$ for any sufficiently small $0<t\le T$.
\end{thm} 

\begin{proof} 
The existence of the solution $w$ is straightforward, and 
\begin{equation} \label{ofS1} 
w (t, 0) =0 
\end{equation} 
for all $0 \le t \le T$.
For the rest of the proof, we consider for simplicity $0\le  y \le 1$, the extension to $-1\le  y \le 0$
will be clear. 
Set $f (t, y) = \partial _y h  (t, y)$ and $v (t, y)= \partial _y w (t, y)$, so that 
$v, v_t , f\in C ([0,T] \times [0,1] )$.
Differentiating equation~\eqref{fODE4}  with respect to $y$ yields
\begin{equation} \label{osfF02}
 v_t = \lambda \frac {\alpha +2} {2}  | w |^\alpha v +  \lambda \frac {\alpha } {2}  | w |^{\alpha -2} w ^2  \overline{v} +  f ,
\end{equation} 
pointwise in $[0,T] \times [0,1] $. Integrating~\eqref{osfF02} and setting
\begin{equation} \label{osfF02b1}
g =   \lambda   \frac {\alpha } {2}  | w |^{\alpha -2} w ^2  \overline{v} +  f 
\end{equation} 
we obtain
\begin{equation} \label{osfF03}
v(t, y )= e^{A(t, y )}   \DIb ' ( y )+ \int _0^t e^{A(t, y )- A(s, y )} g(s, y )\,ds,
\end{equation} 
pointwise on $[0,T] \times [0,1]$, where 
\begin{equation} \label{osfF04}
A(t, y )=  \lambda  \frac {\alpha +2} {2}  \int _0^t  | w (\sigma , y)|^\alpha d\sigma .
\end{equation} 
We note that $A(t,0)= 0$ by~\eqref{ofS1}  so that
\begin{equation} \label{osfF03b1}
v(t,0)=   \DIb '  ( 0)+ \int _0^t  g(s,0)\,ds,
\end{equation} 
for all $0\le t \le T$. 
Given $y>0 $, it follows from~\eqref{osfF03} and~\eqref{osfF03b1} that
\begin{multline} \label{osfF06}
v(t, y ) - v(t,0) = e^{A(t, y )} [ \DIb '  ( y) -  \DIb ' (0)] + [e^{A(t, y )} -1] \DIb ' ( 0 ) \\ + \int _0^t  \Bigl( e^{A(t, y )- A(s, y )}[  g(s, y )  - g (s, 0)] + [e^{A(t, y )- A(s, y )} -1] g(s, 0) \Bigr) \,ds.
\end{multline} 
Observe that, by assumption, $\DIb (  0)=0$ and  
\begin{equation} \label{osfFG09}
z \Eqdef \DIb ' ( 0) \not = 0.
\end{equation} 
Thus for every $\varepsilon \in (0,1)$ there exists $0 <\delta (\varepsilon ) \le \min\{  T, \varepsilon  \}$ such that 
\begin{equation} \label{osfFG10}
 |  v  (t, y ) -z| \le \varepsilon 
\end{equation} 
in the region 
\begin{equation}  \label{osfFG11}
\Delta _\varepsilon = \{ (t,y)\in [0,T] \times [0,1]; \, 0\le t\le \delta (\varepsilon ),\,   0 \le y \le  \delta (\varepsilon ) \}.
\end{equation} 
Estimates~\eqref{ofS1} and~\eqref{osfFG10} yield
\begin{equation} \label{osfFG12}
 | w (t , y) -z y| \le \varepsilon  y
\end{equation} 
in $\Delta _\varepsilon $. 
Next, we observe that by~\eqref{osfF04}, the inequality $ |\,  |z_1|^\alpha - |z_2|^\alpha |\le   |z_1 -z_2|^\alpha $  and~\eqref{osfFG12} 
\begin{equation} \label{ofNs2} 
\begin{split} 
  \Bigl| A(t,y) - \lambda \frac {\alpha +2} {2} t  |z|^\alpha y^\alpha  \Bigr|
& \le  |\lambda |  \frac {\alpha +2} {2} \int _0^t  | \,| w (\sigma  , y)|^\alpha -  |yz|^\alpha   | \, d\sigma\\
& \le  |\lambda |  \frac {\alpha +2} {2} \int _0^t  | w (\sigma , y)- yz |^\alpha  \, d\sigma  \\ & \le 
|\lambda |  \frac {\alpha +2} {2} t \varepsilon ^\alpha  y ^\alpha  
\end{split} 
\end{equation} 
in $\Delta _\varepsilon $. 
In particular, $A$ is bounded. Since $\DIb$ is $C^2$, it follows that
\begin{equation} \label{osfF08b2}
 |e^{A(t, y )} [  \DIb '  ( y ) -  \DIb '  ( 0)] | \le C  y
\end{equation} 
in $\Delta _\varepsilon $.
Moreover,
\begin{equation}  \label{ofNs3} 
|  e^{A(t, y )} -1 - A(t,y) | \le C  | A (t, y)|^2 \le C t^2 y^{2\alpha }
\end{equation} 
in $\Delta _\varepsilon $, where we used~\eqref{ofNs2} in the last inequality.
It follows from~\eqref{ofNs3} and~\eqref{ofNs2} that
\begin{equation} \label{ofcinq16} 
 |[e^{A(t, y )} -1]  \DIb ' ( 0 )| \ge   |\lambda | \frac {\alpha +2} {2}   |z| ( |z|^\alpha -\varepsilon ^\alpha ) t y^\alpha - C t^2 y^{2\alpha }.
\end{equation} 
Next, using again the boundedness of $A$ in $\Delta _\varepsilon $, we 
deduce from~\eqref{fODE5} that
\begin{equation}  \label{osfF08b4}
 \Bigl| \int _0^t  e^{A(t, y )- A(s, y )}[  f(s, y )  - f(s, 0)] \Bigr| \le  C  y^{\beta }
\end{equation} 
in $\Delta _\varepsilon $.
Using~\eqref{ofNs2}, we see that
\begin{equation} \label{osfF08b5}
 | e^{A(t, y )- A(s, y )} | \le 1 + C t y^\alpha 
\end{equation}  
for $0<s<t$ and $(t,y)\in \Delta _\varepsilon $. Moreover, it follows from~\eqref{osfFG10} and~\eqref{osfFG12} that
\begin{equation} \label{osfF08b6}
|  | w |^{\alpha -2} w ^2  \overline{v}  | \le ( |z|+ \varepsilon )^{\alpha +1} y^\alpha 
\end{equation} 
in $\Delta_\varepsilon $. We deduce from~\eqref{osfF02b1}, \eqref{osfF08b4}, \eqref{osfF08b5}   and~\eqref{osfF08b6} that
\begin{equation} \label{osfF08b7} 
\begin{split} 
 \Bigl| \int _0^t  e^{A(t, y )- A(s, y )}[  g(s, y )  - g(s, 0)] \Bigr| & \le  C  y^{\beta }
 \\ & +  |\lambda | \frac {\alpha } {2} (1+Cty^\alpha ) ( |z|+ \varepsilon )^{\alpha +1} t y^\alpha . 
\end{split} 
\end{equation} 
Next, 
\begin{equation*} 
 \begin{split} 
 |e^{A(t, y )- A(s, y )} -1| & \le C | A(t, y )- A(s, y ) | \le C \int _s ^t  | w (\sigma , 0, y )|^\alpha \\
& \le C (t-s)  |y|^\alpha  
 \end{split} 
\end{equation*} 
where we used~\eqref{osfFG12} in the last inequality. Since $g$ is bounded, it follows that
\begin{equation}  \label{osfF08b5b}
 \Bigl| \int _0^t   [e^{A(t, y )- A(s, y )} -1] g(s, 0)  \,ds \Bigr| \le C t^2  y^\alpha .
\end{equation}  
Finally, we observe that the various terms in the right-hand side of~\eqref{osfF06} are estimated by~\eqref{osfF08b2}, \eqref{ofcinq16}, \eqref{osfF08b7}   and~\eqref{osfF08b5b}, and we deduce that
\begin{equation*} 
\begin{split} 
\frac { |v(t, y ) - v(t,0)| } {y^\alpha } & \ge   |\lambda | \frac {\alpha +2} {2}   |z| ( |z|^\alpha -\varepsilon ^\alpha ) t   - C t^2 y^{\alpha } - Cy - C y^{\beta -\alpha } \\
 & - |\lambda | \frac {\alpha } {2} (1+Cty^\alpha ) ( |z|+ \varepsilon )^{\alpha +1} t   - C t^2 
\end{split} 
\end{equation*} 
in $\Delta _\varepsilon $. 
It follows that
\begin{equation*} 
\liminf  _{ y\downarrow 0 } \frac { |v(t, y ) - v(t,0)| } {y^\alpha }   \ge t  |\lambda |  \Bigl( \frac {\alpha +2} {2}   |z| ( |z|^\alpha -\varepsilon ^\alpha )  -  | \frac {\alpha } {2}  ( |z|+ \varepsilon )^{\alpha +1}   \Bigr) - Ct^2 .
\end{equation*} 
Choosing $\varepsilon >0$ and $t>0$ sufficiently small, we see that 
\begin{equation*} 
\liminf  _{ y\downarrow 0 } \frac { |v(t, y ) - v(t,0)| } {y^\alpha }   \ge t  |\lambda |  \frac { |z|^{\alpha +1}} {2}
\end{equation*}
from which estimate~\eqref{fODE6} follows.
\end{proof} 
\begin{rem} 
The assumption that $\DIb \in C^2 ([0, 1], \C )$ is used only once in the proof, see~\eqref{osfF08b2}. It could be replaced by the weaker condition $\DIb \in C^{1, \mu }([0,1], \C )$ with $\alpha <\mu <1$.
\end{rem} 

\section{Semilinear equations as perturbations of an ODE} \label{sAppl} 
In this section we show that Theorem~\ref{eODE2} easily implies Theorems~\ref{eT3} and~\ref{eNLS}.

\begin{proof} [Proof of Theorem~$\ref{eT3}$]
We recall that if $\Delta \DI \in \Cz$, $\DI \in C^3 (\R^N ) \cap W^{3, \infty }( \R^N )$, and $  |\partial ^\gamma u|_\alpha <\infty  $ for all multi-indices $\gamma $ such that $ |\gamma |=3$, then $u$ is once continuously differentiable with respect to $t$, three times continuously differentiable with respect to $x$,  $u_t$ is $\frac {\alpha } {2}$-H\"older continuous in $t$ and 
$\sup _{ 0\le t\le T }  |\partial ^\gamma u(t) |_\alpha <\infty $ for $ |\gamma |=3$ and $0<T<\Tma$. See Theorem~\ref{Holder} below for a precise statement. 
We write the variable in $\R^N $ in the form $x= (x',  y )$, $x'\in \R^{N-1}$, $y \in \R$.
Accordingly, we write $\DI (x)= \DI ( x', y )$ and $u(t, x)= u (t, x', y )$.

Let $\DI \in \Cz$ with $\Delta \DI \in \Cz$.
Suppose further that $ \DI ( x', -y ) = -\DI (x', y ) $ for all $x' \in \R^{N -1} $ and $y \in \R$, 
 and $ \partial  _y \DI (0,0) \not = 0$. 
Note that $u$ inherits the anti-symmetry of the initial condition, i.e. $u (t, x', -y ) \equiv - u(t, x', y )$
for all $0< t< \Tma$. Moreover, there exists $0< t_0< \Tma$ such that $\partial _y u (t, 0, 0)\not = 0$
for all $0\le t\le t_0$. 
Thus we see that for all $0\le t\le t_0$, $u(t)$ satisfies the same assumptions as $\DI$. 
Therefore, it suffices to prove~\eqref{feT3} for $s=0$. 
Assume by contradiction that
\begin{equation}  \label{fS2} 
\int _0 ^T  | \partial _y \Delta  u(t )  |_\beta  < \infty 
\end{equation} 
for some $\alpha <\beta \le 1$ and  $0<T < t_0$.
We apply Theorem~\ref{eODE2} with   $w (t, y ) = u (t, 0, y)$ and $h (t, y)= \Delta u (t, 0, y)$. 
The anti-symmetry property of $u$ implies that $w (t, 0) = h(t,0) =0$ for all $0\le t<\Tma$.
Moreover, it follows from~\eqref{fS2} that the assumption~\eqref{fODE5} is satisfied.
Therefore,  we deduce from~\eqref{fODE6}    
that if $t_0>0$ is sufficiently small, then 
\begin{equation*} 
\limsup  _{  |y| \to 0 } \frac { | \partial _y w (t, y ) - \partial _y w (t,0)|} { |y|^\alpha }>0
\end{equation*} 
for $0<t <t_0$. 
 Since $\partial _y w (t, y)= \partial _y u (t,0, y)$ is $C^1$ in $y $, this yields a contradiction.
The result follows, since we can choose $\DI \in C^\infty _\Comp (\R^N )$ as above.
\end{proof} 

\begin{rem} 
If $\lambda \in \R \setminus \{0\}$, then the statement of Theorem~\ref{eT3} can be improved in the sense that there exists an initial value $\DI \in C^\infty _\Comp (\R^N )$ 
for which~\eqref{feT3}  holds for all $0\le s<t < \Tma$. (We do not require that $t$ is small.)
Indeed, let $\DI \in \Cz$ with $\Delta \DI \in \Cz$.
Suppose further that $ \DI ( x', -y ) = -\DI (x', y ) $ for all $x' \in \R^{N -1} $ and $y \in \R$, $\DI (x', y) \ge 0$ for all $x' \in \R^{N -1} $ and $y >0$, and $ \partial  _y \DI (0,0) >0$. 
Since $u$ inherits the anti-symmetry of the initial condition, i.e. $u (t, x', -y ) \equiv - u(t, x', y )$,
it follows that, restricted to the open half space $\R^N _+ = \R^{N-1} \times (0,\infty )$, $u$ is a solution of the Dirichlet initial value problem on $\R^N _+ $. In particular,  $u(t,x', y ) > 0$ for $y > 0$, and $\partial _y  u (t, x', 0) > 0$.
Thus we see that for all $0\le t<\Tma$, $u(t)$ satisfies the same assumptions as $\DI$, and we can conclude as above.
\end{rem} 

Ne turn next to Theorem~\ref{eNLS}. 
Equation~\eqref{NLSb} is a particular case of the following nonlinear complex Ginzburg-Landau  equation
\begin{equation} \label{GL} 
\begin{cases} 
u_t = e^{i\theta } \Delta u +  \lambda  |u|^\alpha u \\ u(0)= \DI
\end{cases} 
\end{equation} 
in $\R^N $, where $-\frac {\pi } {2}\le   \theta \le \frac {\pi } {2}$, and $0 < \alpha <1$.
Theorem~\ref{eNLS} is therefore a consequence of the following result.

\begin{thm} \label{eGL} 
Let $0<\lambda <1$,  $\lambda \in \C \setminus \{0\}$ and  $-\frac {\pi } {2}\le   \theta \le \frac {\pi } {2}$. Suppose $s> 3+ \frac {N} {2} + \alpha $.
There exists $\DI \in C^\infty _\Comp (\R^N )$ such that there is no $T>0$ for which there exists a solution $u\in C([0,T], H^s (\R^N ) )$ of~\eqref{GL}.
\end{thm}

\begin{rem} \label{eRs12} 
\begin{enumerate}[{\rm (i)}] 

\item \label{eRs12:2}
Suppose $s> 3 + \frac {N} {2}$ (so that $H^s (\R^N ) \subset C^3(\R^N ) \cap  W^{3, \infty } (\R^N ) $). If  $u\in C([0,T], H^s (\R^N ) )$, then $\Delta u,    |u|^\alpha u \in C([0,T] \times \R^N ) \cap C([0,T], L^2 (\R^N ) )$ and $\nabla  \Delta u, \nabla  ( |u|^\alpha u )\in C([0,T] \times \R^N )$. Therefore, equation~\eqref{GL} makes sense for such a $u$. Furthermore,  $u_t \in  C([0,T] \times \R^N ) \cap C([0,T], L^2 (\R^N ) )$ and $\nabla u_t \in  C([0,T] \times \R^N )$. 

\item \label{eRs12:3}
 It follows from the preceding observation that it makes sense to talk of a solution of~\eqref{GL} in  $C([0,T], H^s (\R^N ) )$ if $s> 3 + \frac {N} {2}$. Such a solution satisfies the integral equation
\begin{equation*} 
u(t)= e^{it \Delta }\DI + i \lambda \int _0^t e^{i(t-s) \Delta }( |u|^\alpha u) (s)\, ds.
\end{equation*} 
Using the embedding $H^s (\R^N ) \hookrightarrow L^\infty  (\R^N ) $, it follows easily that such a solution is unique.
In particular, if $\DI $ is anti-symmetric in the last variable, then so is any solution $u\in C([0,T], H^s (\R^N ) )$ of~\eqref{GL}. 

\end{enumerate} 
\end{rem} 

\begin{proof} [Proof of Theorem~$\ref{eGL}$]
We write the variable in $\R^N $ in the form $x= (x',  y )$, $x'\in \R^{N-1}$, $y \in \R$.
Accordingly, we write $\DI (x)= \DI ( x', y )$ and $u(t, x)= u (t, x', y )$.

Let $\DI \in C^\infty _\Comp (\R^N )$ and 
suppose  that $ \DI ( x', -y ) = -\DI (x', y ) $ for all $x' \in \R^{N -1} $ and $y \in \R$, and $ \partial  _y \DI (0,0) >0$. 
Assume by contradiction that there exist $T>0$, $s> 3+\frac {N} {2}+ \alpha $, and a solution $u\in C([0,T], H^s (\R^N ) )$ of~\eqref{GL}.  
As observed in Remark~\ref{eRs12}~\eqref{eRs12:3}, it follows in particular that 
 $u (t, x', -y ) \equiv - u(t, x', y )$ for all $0\le t\le T$.
We apply Theorem~\ref{eODE2} with 
 $w (t, y ) = u (t, 0, y)$ and $h (t, y)= e^{i\theta } \Delta u (t, 0, y)$. 
The regularity assumptions on $w$ and $h$ are satisfied by Remark~\ref{eRs12}~\eqref{eRs12:2}.
The anti-symmetry property of $u$ imply that $w (t, 0) = h(t,0) =0$ for all $0\le t<\Tma$.
Moreover, it follows from Sobolev's embedding theorem (see~\eqref{fSob2}) that
\begin{equation*}
 | \partial _y h(t, y) - \partial _y h(t, 0)| \le C  |y|^{s-3-\frac {N} {2}}
\end{equation*} 
for all $t\in [0,T]$ and $y\in \R$, so that the assumption~\eqref{fODE5} is satisfied.
Therefore,  we deduce from~\eqref{fODE6}    
that if $t_0>0$ is sufficiently small, then 
\begin{equation*} 
\limsup  _{  |y| \to 0 } \frac { |v(t, y ) - v(t,0)|} { |y|^\alpha }>0
\end{equation*} 
for $0<t <t_0$. 
 Since $v (t, y)= \partial _y u (t,0, y)$ is $C^1$ in $y $ by Remark~\ref{eRs12}~\eqref{eRs12:2}, this yields a contradiction.
\end{proof} 

\begin{rem} \label{eRemSD} 
Observe that if $\DI $ is as in the  proof of either Theorem~\ref{eT3} or Theorem~\ref{eGL}, 
then so is $\varepsilon \DI$ for all $\varepsilon \not = 0$.
\end{rem} 

\section{Time-pointwise lack of regularity: the heat equation}
\label{sLast} 

In this section we prove Theorems~\ref{eRN2} and~\ref{eIP1}.
As motivation for the proof of Theorem~\ref{eRN2},  consider an initial value $\DI \in C^\infty _\Comp (\R^N )$ which is odd with respect to the $N^{th}$ variable $y$ and 
 such that $\DI (x', y) =C(x') y$ for small $ |y|$, 
and  write $ | u(t,x', y) |^\alpha u(t,x', y) = C(x') \gamma (t)  |y|^\alpha y +  \widetilde{w} (t,x', y)$ where $ |  \widetilde{w} (t,x', y) |\le C  |y|^{\alpha +2}$. 
This decomposition makes it possible to explicitly calculate $\partial _y^5 [e^{(t-s) \Delta }  |u|^\alpha u] _{ y=0 }$. Incorporating this result into the integral~\eqref{fDU1} yields the desired property. 

The following two lemmas show explicitly how these ideas are implemented.

\begin{lem} \label{eCal1} 
Let $\psi \in C (\R )$ such that $ |\psi (x)|\le C (1+  |x|^m)$ for some $m\ge 0$. Let $\sigma >0$ and set
\begin{equation} 
z = \Ec{\frac {\sigma } {4} } \psi \in C^\infty (\R).
\end{equation} 
It follows that
\begin{equation} \label{fF01b} 
 \partial _x^{5} z(0)  = 8 \pi ^{-\frac {1} {2}} \sigma ^{-3} \int _\R e^{- { y^2}}  [ 
15 - 20 y^2 + 4 y^4 ] (y \sqrt \sigma ) \psi  (y \sqrt \sigma ) dy .
\end{equation} 
\end{lem} 

\begin{proof} 
We have
\begin{equation*} 
z(x)= (\pi \sigma )^{-\frac {1} {2}} \int _\R e^{-\frac { (x-y)^2} {\sigma }} \psi (y) dy,
\end{equation*}  
so that 
\begin{equation} \label{fFo1} 
\partial _x^{n} z(x)= (\pi \sigma )^{-\frac {1} {2}} \int _\R \partial _x^n (e^{-\frac { (x-y)^2} {\sigma }}) \psi (y) dy.
\end{equation} 
We next calculate
\begin{equation*} 
\partial _x^5 (e^{-\frac { (x-y)^2} {\sigma }}) =- e^{-\frac { (x-y)^2} {\sigma }}  \Bigl[ 
 \frac {120 (x-y)} {\sigma ^3} -  \frac {160  (x-y)^3} {\sigma ^4} + \frac {32 (x-y)^5} {\sigma ^5} \Bigr],
\end{equation*} 
so that 
\begin{equation} 
\partial _x^5 (e^{-\frac { (x-y)^2} {\sigma }}) _{ |x=0 } = 8 e^{-\frac { y^2} {\sigma }} \sigma ^{-3} y \Bigl[ 
15 -  \frac {20 y^2} {\sigma } + \frac {4 y^4} {\sigma ^2} \Bigr]. 
\end{equation} 
Thus we deduce from~\eqref{fFo1}  that
\begin{equation} 
 \partial _x^{5} z(0)  = 8 \pi ^{- \frac {1} {2}} \sigma ^{-\frac {7} {2}} \int _\R e^{-\frac { y^2} {\sigma }} \Bigl[ 
15 -  \frac {20 y^2} {\sigma } + \frac {4 y^4} {\sigma ^2} \Bigr] y \psi (y) dy ,
\end{equation} 
from which~\eqref{fF01b} follows.
\end{proof} 

\begin{lem} \label{eCal2} 
If $\psi (x)=  |x|^\alpha x$ with $\alpha >0$, then
\begin{equation} \label{fF1} 
 \partial _x^{5} [ \Ec{\frac {\sigma } {4} } \psi ]  _{ |x=0 }=
- C_\alpha  \sigma ^{-2+ \frac {\alpha } {2}}  ,
\end{equation} 
for all $\sigma >0$, where
\begin{equation} \label{fCalpha} 
C_\alpha = \pi ^{-\frac {1} {2}}  \frac {32\alpha (2- \alpha )} {(\alpha  +3) (\alpha +5)}  \int _\R e^{-y^2 } |y|^{\alpha +6} .
\end{equation} 
\end{lem} 

\begin{proof} 
Considering~\eqref{fF01b}, we must calculate
\begin{multline} \label{feCal2u} 
 \int _\R e^{- { y^2}}  [ 15 - 20 y^2 + 4 y^4 ] |y \sqrt \sigma |^{\alpha +2} dy \\ =
\sigma ^{1+ \frac {\alpha } {2}} \int _\R e^{- { y^2}}  [ 15 - 20 y^2 + 4 y^4 ] |y  |^{\alpha +2} dy .
\end{multline} 
Note that, given any $\beta \ge 0$,
\begin{equation*} 
\int _\R e^{-y^2 } |y|^\beta  = \int _\R e^{-y^2 }  \Bigl( \frac { |y|^\beta y} {\beta +1} \Bigr)'
= \frac {2} {\beta +1} \int _\R e^{-y^2 } |y|^{\beta +2} .
\end{equation*} 
It follows that
\begin{equation*} 
\int _\R e^{-y^2 } |y|^{\alpha +2}   =  \frac {2} {\alpha  +3} \int _\R e^{-y^2 } |y|^{\alpha +4}
=  \frac {4} { (\alpha  +3) (\alpha +5) } \int _\R e^{-y^2 } |y|^{\alpha +6},
\end{equation*} 
and
\begin{equation*} 
\int _\R e^{-y^2 } |y|^{\alpha +4}   =  \frac {2} {\alpha  +5} \int _\R e^{-y^2 } |y|^{\alpha +6}.
\end{equation*} 
Therefore,
\begin{equation}  \label{feCal2d} 
 \int _\R e^{- { y^2}}  [ 15 - 20 y^2 + 4 y^4 ] |y  |^{\alpha +2} dy   = \frac {4\alpha (\alpha -2)} {(\alpha  +3) (\alpha +5)}
  \int _\R e^{-y^2 } |y|^{\alpha +6} .
\end{equation} 
The result now follows from~\eqref{fF01b}, \eqref{feCal2u} and~\eqref{feCal2d}.  
\end{proof} 

\begin{proof} [Proof of Theorem~$\ref{eRN2}$]
We write the variable in $\R^N $ in the form $x= (x',  y )$, $x'\in \R^{N-1}$, $y \in \R$.
Accordingly, we write $\DI (x)= \DI ( x', y )$ and $u(t, x)= u (t, x', y )$.
We note that
\begin{equation} \label{fLapl} 
e^{ t \Delta } = e^{ t \Delta  _{ x' }}  e^{ t \partial _y^2 }
\end{equation} 
where $e^{ t \Delta  _{ x' }} $ is the convolution in $\R^{N -1}$ with the kernel $(4\pi t)^{-\frac {N-1} {2}} e^{- \frac { |x' |^2} {4t}}$ and $e^{ t \partial _y^2 } $ is the convolution in $\R $ with the kernel $(4\pi t)^{-\frac {1} {2}} e^{- \frac { y^2} {4t}}$.

Let $\DI \in C^\infty _\Comp (\R^N ) $ and let $u\in C([0, \Tma), \Czu \cap L^1 (\R^N )  )$ be the corresponding maximal solution of~\eqref{NLH}.  
Assume that 
\begin{equation} \label{NH2} 
\DI (x', -y) \equiv -\DI (x', y ) \text{ and } \partial _y \DI( 0 , 0 ) \not =  0  .
\end{equation}
Recall that $u$ is $C^1$ in time and $C^3 $ in space, and that $\sup _{ 0\le t\le  T}  \| u (t) \| _{ W^{2, \infty } } <\infty $ for all $0<T<\Tma$. (See e.g. Theorem~\ref{Holder}.)
Under the assumption~\eqref{NH2} it follows that for all $0\le t\le T$ and $x'\in \R^{N-1} $, $u(t, x', \cdot )$ is odd 
and  $u(t,x', y ) >0$ for $y>0$. 
Moreover, if we set
\begin{gather} 
\eta _0 =  \partial _y \DI( 0 , 0 )  \\
\eta (t, x')= \partial _{y } u(t,x' ,0),\quad 0\le t\le T, x'\in \R^{N-1} 
\end{gather} 
then $\eta \in C([0, \Tma ) \times \R^{N-1} )$.
Therefore, it follows from~\eqref{NH2} that for every $0<\varepsilon < 1 $ there exists $0< \delta _\varepsilon  <\Tma$ such that
\begin{equation} \label{fEta0N} 
\sup _{\substack{  0\le s\le \delta _\varepsilon \\  |x'|\le \delta _\varepsilon }}  |  \eta (s, x')- \eta_0 | \le \varepsilon  |\eta_0|.
\end{equation} 
Let $w$ be defined by
\begin{equation} \label{fEtawuN} 
u(s,x', y )= \eta (s, x') y  + w(s,x', y ) .
\end{equation} 
We claim that
\begin{equation}  \label{fEtawdN} 
 |w(s,x', y )| \le C  y  ^2, 
\end{equation} 
for all $s\in [0,T]$, $x'\in \R^{N-1}$ and $y\in \R$, where $C= \frac {1} {2} \sup _{ 0\le t\le  T}  \| u (t) \| _{ W^{2, \infty } }$. 
Indeed, fix $0\le s\le T$ and $x'\in \R^{N-1}$, and set $h(y)= w(s,x',y)=u(s, x', y)- \eta (s, x') y$.
We have  $h'(y)= \partial _y u ( s,x', y) -  \eta (s, x')$ and $h'' (y)= \partial ^2 _y u (s, x' ,y)$. In particular, $h(0)= h'(0)= 0$ so that
\begin{equation*} 
h (y) = \int _0 ^y \int _0^\tau  \partial ^2 _y u (s, x' ,\sigma )\, d\sigma \, d\tau ,
\end{equation*} 
and so
\begin{equation*} 
 | w( s, x', y )|=|h(y)| \le \frac {1} {2} y^2  \| u(s, \cdot ,\cdot ) \| _{ W^{2, \infty } } ,
\end{equation*} 
which proves~\eqref{fEtawdN}. 
Since $u$ is bounded, we deduce easily from~\eqref{fEtawuN}-\eqref{fEtawdN}   that 
\begin{equation} \label{fF3N} 
[ |u|^\alpha u ] (s,x', y  )=  |\eta (s, x') y|^{\alpha } \eta (s, x') y  +  \widetilde{w} (s,x', y ), 
\end{equation} 
with 
\begin{equation} \label{fF4N} 
 | \widetilde{w} (s,x', y )| \le C  |y |^{\alpha +2},
\end{equation} 
for all $s \in [0,T]$ and $(x', y ) \in \R^N $. 
It follows from Lemma~\ref{eCal2} that if
 $\psi (y )=  |y |^\alpha y $, then 
\begin{equation} \label{fF1N} 
 \partial _{y }^{5} [ e^{  {\frac {\sigma } {4} } \partial _y^2  }\psi ]  _{ |y=0 }= 
-   C_\alpha  \sigma ^{-2+ \frac {\alpha } {2}},
\end{equation} 
for all $\sigma >0$, where $C_\alpha > 0$ is given by~\eqref{fCalpha}. 
We deduce from~\eqref{fLapl} and~\eqref{fF1N} that 
\begin{equation} \label{fTE1} 
\begin{split} 
\partial _y^{5} [ e^{ {(t-s) } \Delta } & (   | \eta (s, x')y |^\alpha \eta (s, x') y ) ] _{ |y=0 } 
\\ & = -  C_\alpha (4(t-s) )^{-2+ \frac {\alpha } {2}} e^{ (t-s) \Delta  _{ x' }} [ | \eta (s, x') |^\alpha  \eta (s, x')] .
\end{split} 
\end{equation} 
On the other hand, it follows from~~\eqref{fEta0N} that there exists $C$ independent of $0<\varepsilon <1$ such that
\begin{equation*} 
\sup _{\substack{  0\le s\le \delta _\varepsilon \\  |x'|\le \delta _\varepsilon }}  |  | \eta (s, x')|^\alpha  \eta (s, x')-  |\eta_0|^\alpha \eta_0 | \le  \varepsilon C .
\end{equation*} 
By possibly choosing $\delta _\varepsilon >0$ smaller, we deduce that 
\begin{equation} \label{fTE2} 
\sup _{\substack{  0\le s < t\le \delta _\varepsilon \\  |x'|\le \delta _\varepsilon }}  |  e^{ (t-s) \Delta  _{ x' }} [  | \eta (s, x')|^\alpha  \eta (s, x') ] -  |\eta_0|^\alpha \eta_0 | \le  \varepsilon C .
\end{equation} 
It follows from~\eqref{fTE1} and~\eqref{fTE2} that 
\begin{equation} \label{fF2N} 
\begin{split} 
 |\partial _y^{5} [ e^{ {(t-s) } \Delta }  |\eta (s, x') y |^\alpha \eta (s, x') y )]  _{ |y=0 }
 +  C_\alpha (4(t-s) )^{-2+ \frac {\alpha } {2}}  & |\eta_0|^\alpha \eta_0 | \\ &
\le \varepsilon C (t-s)^{-2 + \frac {\alpha } {2}}
\end{split} 
\end{equation} 
for all $0\le  s<t \le \delta _\varepsilon $ and $  |x'|\le \delta _\varepsilon $.
On the other hand, it follows from~\eqref{fF01b}  that
\begin{multline*} 
 |\partial _y^{5} [ e^{ {(t-s) } \Delta } \widetilde{w} (s, x', y)]  _{ |y=0 } | = 
 |e^{ {(t-s) } \Delta  _{ x' }}
 \partial _y^{5} [ e^{ {(t-s) } \partial _y^2} \widetilde{w} (s, x', y)]  _{ |y=0 } | \\ \le C (t-s) ^{-\frac {5} {2}}
 e^{ {(t-s) } \Delta  _{ x' }}  \Bigl(  \int _\R e^{- { y^2}}   
 |15 - 20 y^2 + 4 y^4 | \,  |y|  |  \widetilde{w}( s, x' , y \sqrt{4(t-s)})   | dy \Bigr) \\
  \le C (t-s) ^{-\frac {5} {2}}  e^{ {(t-s) } \Delta  _{ x' }}  \Bigl(  \int _\R e^{- \frac {y^2} {2}}   
  |  \widetilde{w}( s, x', y \sqrt{4(t-s)})   | dy \Bigr).
\end{multline*} 
Applying~\eqref{fF4N}, we deduce that
\begin{equation} \label{fF5N} 
 |\partial _y^{5} [ e^{ {(t-s) } \Delta } \widetilde{w} (s, x', y)]  _{ |y=0 } | 
\le C (t-s) ^{-\frac {3} {2} + \frac {\alpha } {2}} 
  \int _\R e^{- \frac {y^2} {2}}    |y|^{\alpha +2} \le C  (t-s) ^{-\frac {3} {2} + \frac {\alpha } {2}} .
\end{equation} 
It now follows from~\eqref{fF3N}, \eqref{fF2N} and~\eqref{fF5N} that  
\begin{equation} \label{fF2Nb1} 
\begin{split} 
 |\partial _y^{5} [ e^{ {(t-s) } \Delta } & |u (s, x', y )|^\alpha u  (s, x', y ) ]  _{ |y=0 }
\\ & +  C_\alpha (4(t-s) )^{-2+ \frac {\alpha } {2}}   |\eta_0|^\alpha \eta_0 | 
\le \varepsilon C (t-s)^{-2 + \frac {\alpha } {2}}
\end{split} 
\end{equation} 
for all $0\le  s<t \le \delta _\varepsilon $ and $  |x'|\le \delta _\varepsilon $.

The  point is that $ (t-s) ^{-2+ \frac {\alpha } {2}}  $ is  not integrable in $s$ at $s=t$ under the assumption $0< \alpha <2$. 
We now conclude the proof as follows.
Fix $0< \tau  < \Tma $ and, given any $0\le t< \tau $, set
\begin{equation} \label{fF7N} 
\NH  (t,  \tau , x', y )= \int _0 ^{t } \Ec{( \tau   -s)}  |u (s)|^\alpha u (s) \, ds.
\end{equation} 
Since $\tau -s\ge \tau -t>0$ for $s\in [0,t]$,   the smoothing effect of the heat semigroup implies that the integrand in~\eqref{fF7N} is integrable as a function with values in $H^{m } (\R^N  ) $ for all $m\ge 0$.  Choosing $m$ large enough so  that $H^{m } (\R^N  ) \subset C^5 (\R^N )$, we see that 
the formula
\begin{equation} \label{fF8N} 
\partial _y^5 \NH  (t,  \tau ,x', 0  )= \int _0^{t} \partial _y^5 [ \Ec{( \tau   -s)}  |u (s)|^\alpha u (s) ] _{ |y=0 } \, ds
\end{equation} 
makes sense. 
Applying~\eqref{fF2Nb1} with $t$ replaced by $\tau $, we deduce that
\begin{equation} \label{fF9N} 
\begin{split} 
 | \partial _y^5 \NH  (t,  \tau , x', 0 ) | & \ge [ C_\alpha 4^{-2+ \frac {\alpha } {2}}   |\eta_0|^{ \alpha +1}  -\varepsilon C ] \int _0^t (\tau -s) ^{-2+ \frac {\alpha } {2}} \,ds
 \\ & = \frac {2} {2-\alpha }  [ C_\alpha 4^{-2+ \frac {\alpha } {2}}   |\eta_0|^{ \alpha +1}  -\varepsilon C ]  [ (\tau -t)^{-\frac {2-\alpha } {2}} - \tau ^{-\frac {2-\alpha } {2}}  ] 
\end{split} 
\end{equation} 
for $0<t <\tau  \le \delta _\varepsilon $ and $  |x'|\le \delta _\varepsilon $.
We now fix $s, p$ such that $s\ge 5+ \frac {1} {p}$ and $\varepsilon $ sufficiently small so that $C_\alpha 4^{-2+ \frac {\alpha } {2}}   |\eta_0|^{ \alpha +1}  > \varepsilon C$, and we deduce from~\eqref{fF9N}  and the embedding
 $H^{s,p} (\R ) \hookrightarrow W^{5, \infty }( \R )$  that
\begin{equation} \label{fF10N} 
 \| \NH  (t,   \tau , x' , \cdot ) \| _{ H^{s,p} (\R)} \ge a ( \tau  - t)^{-1+ \frac {\alpha } {2}}
- a  t   ^{-1+ \frac {\alpha } {2}} - A,
\end{equation} 
for some constants $a,A>0$ independent of $0<t <\tau  \le \delta _\varepsilon $ and $  |x'|\le \delta _\varepsilon $.
We now use the property
\begin{equation*} 
 \| u \| _{ L^p _{ x' } ( \{  |x'|<\delta _\varepsilon  \} , H^{s, p}( \R_y) ) } \le  \| u \| _{ L^p _{ x' } (\R^{N-1}, H^{s, p}( \R_y) ) } \le C  \| u \| _{ H^{s , p} (\R^N  _{ x', y }) },
\end{equation*} 
 (see~\eqref{fSob1}), and we deduce from~\eqref{fF10N}  that for some constant $\nu >0$
\begin{equation}  \label{fF10N2b1} 
\nu \|   \NH  (t,   \tau ) \| _{ H^{s,p} (\R^N)} \ge a ( \tau  - t)^{-1+ \frac {\alpha } {2}}
- a  t   ^{-1+ \frac {\alpha } {2}} - A  
\end{equation} 
so that
\begin{equation} \label{fF10N2} 
  \|   \NH  (t,   \tau ) \| _{ H^{s,p} (\R^N)}    \goto  _{  \tau  \downarrow t } \infty.
\end{equation} 
Observe that
\begin{equation*} 
\begin{split} 
 \lambda \NH (t,  \tau , x', y ) & =  \lambda \Ec{(  \tau  -t ) } \int _0^{t  } \Ec{(t -s)}  |u (s)|^\alpha u (s) \, ds \\ &=   \Ec{(  \tau  -t ) } [u(t ) - \Ec{t } \DI]
\end{split} 
\end{equation*} 
hence
\begin{equation*} 
 |\lambda |\,  \| \NH (t,  \tau  ) \| _{ H^{s, p} }\le   \| u(t) \| _{ H^{s,p} }+  \| \DI \| _{ H^{s,p} }.
\end{equation*} 
Note that $\| \DI \| _{ H^{s,p} } <\infty $. Therefore,
letting $\tau \downarrow t$ and 
applying~\eqref{fF10N2}, we conclude that $\| u(t) \| _{ H^{s,p} } =\infty $.
\end{proof} 

\begin{rem} \label{eRemSDb} 
Observe that if $\DI $ is as in the  proof of Theorem~\ref{eRN2},
then so is $\varepsilon \DI$ for all $\varepsilon \not = 0$.
\end{rem} 

\begin{rem} 
A careful analysis of the proof of Theorem~\ref{eRN2} shows that the property $C_\alpha >0$
(i.e., $\alpha <2$), where $C_\alpha $ is given by~\eqref{fCalpha},  could in principle be replaced by the condition $C_\alpha \not = 0$ (i.e.,  $\alpha \not = 2$). 
On the other hand, the condition $\alpha <2$ is crucial in proving~\eqref{fF10N2} by letting $\tau \downarrow t$ in~\eqref{fF10N2b1}.

\end{rem} 

\begin{proof} [Proof of Theorem~$\ref{eIP1}$]
We argue by contradiction. Suppose  that  for some $s\ge 0$ the Cauchy problem~\eqref{NLH} is locally well posed for small data in $H^s (\R^N ) $. 
Using Theorem~\ref{eRN2} and Remark~\ref{eRemSDb} with $p=2$ we see that $2s \le 11$. It follows then from~\eqref{fCOR1} that 
\begin{equation}  \label{LWPu} 
N-2s > \frac {4} {\alpha } . 
\end{equation} 
A scaling argument allows us now to conclude. Indeed, 
let $\varphi \in C^\infty _\Comp (\R^N )$, $\varphi \not = 0$. Since $\lambda >0$, for $k>0$ sufficiently large the (classical) solution $u$ of~\eqref{NLH} with initial value $\DI= k \varphi $ blows up in finite time, say at $\Tma$.
Given $\mu  >0$, let
\begin{equation}  \label{LWPd} 
u_\mu  (t,x) = \mu  ^{\frac {2} {\alpha }} u(\mu  ^2t, \mu  x).
\end{equation} 
It follows that $u_\mu  $ is a solution of~\eqref{NLH} with the initial value
\begin{equation}  \label{LWPt} 
\DI^\mu  (x)= \mu  ^{\frac {2} {\alpha }} \DI (\mu  x),
\end{equation}  
which blows up at 
\begin{equation}  \label{LWPs} 
\Tma ^\mu  = \frac {\Tma } {\mu  ^2} \goto  _{ \mu  \to \infty  }0.
\end{equation} 
On the other hand, 
\begin{equation}  \label{LWPq} 
 \| \DI ^\mu  \| _{ H^s } \le \mu  ^{\frac {2} {\alpha } + s - \frac {N} {2}}  \| \DI \| _{ H^s },
\end{equation} 
for $\mu\ge 1$. Using~\eqref{LWPu}, we see that
\begin{equation} \label{LWPc} 
 \| \DI^\mu  \| _{ H^s } \goto _{ \mu  \to \infty  }0.
\end{equation}  
Comparing~\eqref{LWPc} and~\eqref{LWPs}, we conclude that~\eqref{NLH} cannot be locally well posed for small data in $H^s (\R^N ) $. 
\end{proof}

\appendix

\section{H\"older regularity for the heat equation} \label{sH} 

In this section, we state a classical regularity result for the heat equation~\eqref{NLH}.
For completeness, we give the proof, which is based on classical arguments.

\begin{thm} \label{Holder} 
Let $\alpha >0$, $\lambda \in \C$, $\DI \in \Cz $ and let $u\in C([0, \Tma ), \Cz )$ be the corresponding maximal solution of~\eqref{NLH}. 
Fix $0<  \widetilde{\alpha}  <1$ with $ \widetilde{\alpha } \le \alpha  $,  $0 < T < \Tma$, and assume further that $\Delta \DI\in \Cz$, 
 $\DI \in C^3 (\R^N )$ and
\begin{equation} \label{fRE7b1} 
\sup _{\substack{  |\gamma |\le 3 \\ x\in \R^N }} | \partial ^\gamma \DI (x)|
  + \sup 
 _{  |\gamma |=3}  | \partial ^\gamma \DI | _{ \widetilde{\alpha }  }  < \infty 
\end{equation} 
with the notation~\eqref{fHO1}. 
It follows that $\Delta u\in C([0,\Tma), \Cz)$, that  $\partial _t u$, $\nabla \partial _t u$, and all space-derivatives of $u$  of order $\le 3$ belong to $C([0,T] \times \R^N )$, and that
\begin{equation} \label{fRE2b2} 
\begin{split} 
\sup  _{ \substack{ 2\ell +  |\gamma |\le 3 \\ 0\le t\le T} }   \| \partial _t^\ell \partial _x^\gamma u (t) \| _{L^\infty }    & + 
\sup _{\substack{ x\in \R^N \\ 0\le t<s\le T \\ 2\le 2\ell +  |\gamma |\le 3}}
 \frac { | \partial _t^\ell \partial _x^\gamma u (t,x) - \partial _t^\ell \partial _x^\beta u (s,x)|} { |t-s|^{\frac {  \widetilde{\alpha}   +3-2 \ell - |\gamma |} {2}}} \\ &
 + \sup 
 _{\substack{  0\le t\le T\\ 2\ell +  |\gamma |=3}}
 |  \partial _t^\ell \partial _x^\gamma u  (t ) |_ {\widetilde{\alpha} }  < \infty .
\end{split} 
\end{equation} 
In particular
\begin{equation} \label{fRE2b3} 
\sup 
 _{\substack{ 0\le t\le T\\    |\gamma |=3}}
 |   \partial _x^\gamma u  (t ) |_{ \widetilde{\alpha} }   < \infty .
\end{equation} 
In the above estimates, $\partial _x^\gamma = \partial  _{ x_1 }^{\gamma _1} \cdots  \partial  _{ x_N }^{\gamma _N}$ if $\gamma  = (\gamma  _1, \cdots , \gamma  _N )$. 
\end{thm} 

\begin{proof} 
We define the Laplacian $A$ on $\Cz$ by 
\begin{equation} 
\begin{cases} 
D( A) =\{ u\in \Cz;\, \Delta u\in \Cz \} \\
Au = \Delta u \quad u\in D(A).
\end{cases} 
\end{equation} 
We equip $D(A) $ with the graph norm $ \| u \| _{ D(A) } =  \| u \| _{ L^\infty  } +  \| \Delta u \| _{ L^\infty  }$. 
It follows that $A$ is $m$-dissipative with dense domain and that $D(A) \hookrightarrow C^1_0 (\R^N )$. More precisely
\begin{equation} \label{fRE1} 
 \| \nabla u \| _{ L^\infty  } \le C  \| \Delta u \| _{ L^\infty  }^{\frac {1} {2}}  \| u \| _{ L^\infty  }^{\frac {1} {2}}
\end{equation} 
for all $u\in D(A)$. 
Indeed, consider the Bessel potential $G_\sigma $, $\sigma >0$. (See Aronszajn and Smith~\cite{AronszajnS}.) If $f\in \Cz$, then $u= G_2 \star f$ satisfies $-\Delta u+u=f$, $u\in \Cz$ and $\Delta u\in \Cz$. Moreover, since $  \| G_\sigma  \| _{ L^1 }=1 $ (see~\cite[formula~(4.6$'$), p.~417]{AronszajnS}) we see that if $u\in D(A)$, then $  \| u \| _{ L^\infty  } \le  \| -\Delta u + u \| _{ L^\infty  }$. By an obvious scaling argument, we see that $ \rho   \| u \| _{ L^\infty  } \le  \| -\Delta u + \rho  u \| _{ L^\infty  }$ for all $\rho  >0$. Thus  $A$ is $m$-dissipative. Furthermore, $C^\infty _\Comp (\R^N )\subset D(A)$ so that $D(A)$ is dense. 
To prove~\eqref{fRE1}, consider $u\in D(A)$. We have $u= G_2 \star (-\Delta u+u)$. Since $\nabla G_2 \in L^1 (\R^N ) $ (see~\cite[formula~(4.5), p.~417]{AronszajnS}), we deduce that $ \| \nabla u\| _{ L^\infty  }\le C ( \| \Delta u \| _{ L^\infty  }+  \| u \| _{ L^\infty  })$, from which~\eqref{fRE1} follows by scaling.
 
Let now $\DI \in \Cz $ and let $u\in C([0, \Tma ), \Cz )$ be the corresponding maximal solution of~\eqref{NLH}. It follows from the above observations and from Pazy~\cite[Theorem~1.6, p.~187]{Pazy} that $u \in C^1 ([0,T], \Cz ) \cap C ([0,T], D(A))$.  
In particular, using~\eqref{fRE1}, 
\begin{equation} \label{fRE4} 
\sup  _{ 0\le t\le T } (  \| u(t)  \| _{ L^\infty  }+  \| \nabla u(t)  \| _{ L^\infty  })  + \sup  _{ 0\le t<s \le T } \frac { \| u(t) - u(s) \| _{ L^\infty  }} { |t-s|} <\infty .
\end{equation} 
Setting
\begin{equation} 
f= \lambda  |u|^\alpha u
\end{equation} 
and using the formula
\begin{equation}  \label{fRE3} 
\nabla f = \lambda \frac {\alpha +2} {2}  | u |^\alpha \nabla u +  \lambda \frac {\alpha } {2}  | u |^{\alpha -2} u ^2  \nabla \overline{u} 
\end{equation} 
we deduce from~\eqref{fRE4} that
\begin{equation}  \label{fRE6} 
\begin{split} 
\sup  _{ 0\le t\le T } & ( \| f (t) \| _{L^\infty }+  \| \nabla f(t) \| _{ L^\infty  } )  +
\sup _{ 0\le t<s\le T}
 \frac {  \| f ( t ) - f ( s ) \| _{ L^\infty  } } { |t-s|}
 \\& + \sup 
 _{ 0\le t\le T}
 | f (t )|_1 < \infty .
\end{split} 
\end{equation} 
In particular,
\begin{equation} \label{fRE2b1} 
\sup  _{ 0\le t\le T }   \| f (t) \| _{L^\infty }   + 
\sup _{ 0\le t<s\le T}
 \frac {  \| f ( t ) - f ( s ) \| _{ L^\infty  } } { |t-s|^{\frac {  \widetilde{\alpha}   } {2}}}
 + \sup 
 _{ 0\le t\le T}  | f (t ) |_{ \widetilde{\alpha } }    < \infty .
\end{equation} 
We apply Ladyzhenskaya {\it et al.}~\cite{LSU}, Chapter~IV, Section~2, p.~273, estimate~(2.1) 
to the equations satisfied by the real and the imaginary parts of $u$, 
with $l =  \widetilde{\alpha}  $ and estimate~(2.2) with $l=  \widetilde{\alpha}  +2$. It follows (among other properties) that
\begin{equation}  \label{fRE7} 
\sup _{  0\le t<s\le T }
 \frac {   \| \nabla u(t ) -\nabla u(s ) \| _{ L^\infty  } } { |t-s|^{\frac { \widetilde{\alpha}  +1} {2}}}
  + \sup 
 _{ 0\le t\le T}  |\nabla u (t ) |_1 < \infty .
\end{equation} 
We next recall the elementary estimate 
\begin{equation} \label{fEE1} 
\begin{split} 
 |\,  |z_1|^\alpha -  |z_2|^\alpha | &+  |\,  |z_1|^{\alpha -2} z_1^2 -  |z_2|^{\alpha -2} z_2^2 | 
 \\ & \le 
 \begin{cases} 
C  |z_1 - z_2|^\alpha & 0<\alpha \le 1 \\ 
C (  |z_1|^{\alpha -1} +  |z_2|^{\alpha -1} )  |z_1 - z_2| & \alpha \ge 1
 \end{cases} 
\end{split} 
\end{equation} 
(See e.g.~\cite{CFH}, formulas~(2.26) and~(2.27).)
Estimate~\eqref{fRE6}, formula~\eqref{fRE3}, estimates~\eqref{fEE1}, \eqref{fRE4} and~\eqref{fRE7}  imply that 
\begin{equation} \label{fRE2} 
\begin{split} 
\sup  _{ 0\le t\le T } & ( \| f (t) \| _{L^\infty }+  \| \nabla f(t) \| _{ L^\infty  } )   + 
\sup _{  0\le t<s\le T }
 \frac {  \|f (t ) -f(s ) \| _{ L^\infty  } } { |t-s|^{\frac {1+  \widetilde{\alpha}  } {2}}}\\ &
 + \sup _{  0\le t<s\le T }
 \frac {   \| \nabla f(t ) -\nabla f(s ) \| _{ L^\infty  } } { |t-s|^{\frac { \widetilde{\alpha}  } {2}}}
  + \sup 
 _{ 0\le t\le T}   |\nabla f (t ) |_{ \widetilde{\alpha}}   < \infty .
\end{split} 
\end{equation} 
This allows us to apply again 
Ladyzhenskaya {\it et al.}~\cite{LSU}, Chapter~IV, Section~2, p.~273, but this time we let $l=  \widetilde{\alpha}  +1$ in~(2.1) and $l=  \widetilde{\alpha}  +3$ in~(2.2). 
The conclusion follows.
\end{proof}


\begin{thebibliography}{99} 

\bibitem{AdamsF}{Adams R.A. and Fournier J. J. F.} { {\it Sobolev spaces.
Second edition.} Pure and Applied Mathematics (Amsterdam)  {\bf 140}. Elsevier/Academic Press, Amsterdam, 2003.}
\MScN{MR2424078}

\bibitem{AlazardC}{Alazard T. and Carles R.} {Loss of regularity for supercritical nonlinear Schr\"o\-din\-ger equations,  Math. Ann.  {\bf 343}  (2009),  no. 2, 397--420.}
\MScN{MR2461259} \DOI{10.1007/s00208-008-0276-6}

\bibitem{AronszajnS}{Aronszajn N. and Smith K. T.} {Theory of Bessel potentials. I. 
Ann. Inst. Fourier (Grenoble)  {\bf 11}  1961 385--475.}
\MScN{MR0143935} \LINK{http://www.numdam.org/item?id=AIF_1961__11__385_0}

\bibitem{BerghL}{Bergh J. and L\"ofstr\"om J.} {{\it Interpolation
spaces. An introduction.} Grundlehren der Mathematischen Wissenschaften  {\bf 223} . Springer-Verlag, Berlin-New York, 1976.}
\MScN{MR0482275} \DOI{10.1007/978-3-642-66451-9}

\bibitem{BurqGT}{Burq N., G\'erard P. and Tzvetkov N.} {Multilinear eigenfunction estimates and global existence for the three dimensional nonlinear Schr\"o\-din\-ger  equations,  Ann. Sci. \'Ecole Norm. Sup. (4) {\bf 38}  (2005),  no. 2, 255--301.}
\MScN{MR2144988} \DOI{10.1016/j.ansens.2004.11.003}

\bibitem{Carles}{Carles R.} {Geometric optics and instability for semi-classical
Schr\"o\-din\-ger equations,  Arch. Ration. Mech. Anal.  {\bf 183}  (2007), 
no.~3, 525--553.}
\MScN{MR2278414} \DOI{10.1007/s00205-006-0017-5}

\bibitem{CFH} {Cazenave T., Fang D. and Han Z.} {Continuous dependence for NLS in fractional order spaces, Ann. Inst. H. Poincar\'e Anal. Non Lin\'eaire  {\bf 28} (2011), no. 1, 135--147.}
\MScN{MR2765515} \DOI{10.1016/j.anihpc.2010.11.005}

\bibitem{CWHs}{Cazenave T. and Weissler F. B.} {The
Cauchy problem for the critical nonlinear Schr\"o\-din\-ger equation in
$H^{s}$, Nonlinear Anal.  {\bf 14} (1990), no.~10, 807--836.}
{\MScN{MR1055532}} {\DOI{10.1016/0362-546X(90)90023-A}}

\bibitem{ChristCT}{Christ M., Colliander J. and Tao T.} {Ill-posedness for nonlinear Schr\"o\-din\-ger  and wave equations, arXiv:math/0311048v1 [math.AP] (2003).}
\LINK{http://arxiv.org/abs/math/0311048}

\bibitem{FangH}{Fang D. and Han Z.} {On the well-posedness for NLS in $H^s$.
J. Funct. Anal.  {\bf 264}  (2013), no. 6, 1438--1455.}
\MScN{MR3017270} \DOI{10.1016/j.jfa.2013.01.005}

\bibitem{Kato2}{Kato T.} {On nonlinear Schr\"o\-din\-ger equations, Ann. Inst. H.~Poin\-ca\-r\'e Phys. Th\'eor. {\bf 46} (1987), no. 1, 113--129.}
\MScN{MR0877998} \LINK{http://www.numdam.org/item?id=AIHPA_1987__46_1_113_0}

\bibitem{Kato3}{Kato T.} {Nonlinear Schr\"o\-din\-ger equations, in {\it
Schr\"o\-din\-ger Operators  (S\o nderborg, 1988)}, Lecture Notes in Phys. {\bf 345}, Springer, Berlin, 1989, 218--263.}
\MScN{MR1037322} \DOI{10.1007/3-540-51783-9_22}

\bibitem{Kato1}{Kato T.} {On nonlinear Schr\"o\-din\-ger  equations. II. $H^s$-solutions and unconditional well-posedness. 
J. Anal. Math.  {\bf 67}  (1995), 281--306.}
\MScN{MR1383498} \DOI{10.1007/BF02787794}

\bibitem{LSU}{Ladyzhenskaya O.A., Solonnikov V.A. and Ural'ceva N.N.} {{\it Linear and quasilinear equations of parabolic type.} Translated from the Russian by S. Smith. Translations of Mathematical Monographs, Vol.  {\bf 23}.  American Mathematical Society, Providence, R.I. 1968.}
\MScN{MR0241821}

\bibitem{MolinetRY}{Molinet L., Ribaud F. and Youssfi A.} {Illposedness issues
for a class of parabolic equations, Proc. Royal Soc. Edinburgh Sect.~A  {\bf 132} (2002), 1407--1416.}
\MScN{MR1950814} 

\bibitem{Pazy}{Pazy A.} {{\it Semi-groups of linear operators and
applications to partial differential equations}, 
Applied Math. Sciences {\bf 44}, Springer, New-York, 1983.}
\MScN{MR0710486}

\bibitem{Pecher}{Pecher H.} {Solutions of semilinear
Schr\"o\-din\-ger equations in $H^s$, Ann. Inst. H.~Poin\-ca\-r\'e Phys. Th\'eor. {\bf 67} (1997), no.~3, 259--296.}
\MScN{MR1472820} \LINK{http://www.numdam.org/item?id=AIHPA_1997__67_3_259_0}

\bibitem{Ribaud}{Ribaud F.} {Cauchy problem for semilinear
parabolic equations with initial data in $H^s_p(\R^N )$ spaces, 
Rev. Mat. Ibero\-ame\-ri\-cana  {\bf 14} (1998), no.~1, 1--46.}
\MScN{MR1639271} \DOI{10.4171/RMI/232}

\bibitem{Tao}{Tao T.} {{\it Nonlinear dispersive equations. Local and global
analysis}. CBMS Regional Conference Series in Mathematics {\bf 106}. Published
for the Conference Board of the Mathematical Sciences, Washington, DC; by the
American Mathematical Society, Providence, RI, 2006.}
\MScN{MR2233925}

\end{thebibliography}
\end{document}